\documentclass[11pt]{amsart}
\usepackage{amsfonts}
\usepackage{color}
\usepackage{amsthm}
\usepackage{amsmath}
\usepackage{enumerate}
\usepackage[usenames,dvipsnames]{xcolor}
\usepackage[colorinlistoftodos,prependcaption]{todonotes}
\usepackage{hyperref,xargs}
\usepackage{enumerate}
\usepackage[all]{xy}
\usepackage[T1]{fontenc}
\usepackage[utf8]{inputenc}
\usepackage{amssymb}
\usepackage[polish,english]{babel}
\usepackage[percent]{overpic}
\usepackage{xspace} 
\hypersetup{
  pdfstartview={FitH},
  colorlinks = true,
  linkcolor=blue,
  urlcolor = black,
  citecolor=red, 
  pdfpagemode = UseNone}

\theoremstyle{plain}
\newtheorem{theorem}{Theorem}[section]
\newtheorem{lemma}[theorem]{Lemma}
\newtheorem{proposition}[theorem]{Proposition}

\newtheorem{corollary}[theorem]{Corollary}
\theoremstyle{definition}
\newtheorem{definition}[theorem]{Definition}
\newtheorem{remark}[theorem]{Remark}

\newcommand{\HE}[3]{H_E^{#1}(#2,#3)}
\newcommand{\set}[1]{\left\{#1\right\}}
\newcommand{\cl}[1]{\overline{#1}}
\newcommand{\inv}[1]{\mathrm{Inv}\left(#1\right)}
\newcommand{\inte}[1]{\mathrm{int}\left(#1\right)}
\newcommand{\abs}[1]{\left\lvert{#1}\right\rvert}
\newcommand{\norm}[1]{\left\lVert{#1}\right\rVert}
\newcommand{\ip}[2]{\langle {#1} , {#2} \rangle }

\newcommand{\ev}{\mathrm{ev}}

\def\coef{\mathbb{F}}
\def\:{\colon}
\def\ss{\subset}

\def\d{\cdot}
\def\f{\varphi}
\def\r{\rho}

\def\R{\mathbb R}
\def\N{\mathbb N}
\def\G{\Gamma}
\def\tG{G_{\psi}}

\def\mN{\mathbb{N}}
\def\mR{\mathbb{R}}
\def\mH{\mathbb{H}}
\def\mZ{\mathbb{Z}}
\def\ls{\mathcal{LS}}
\newcommand{\propC}{(C)\xspace}
\def\ch{\mathrm{ch}_E}

\DeclareMathOperator{\invariant}{Inv}
\DeclareMathOperator{\interior}{int}
\DeclareMathOperator{\codim}{codim}
\DeclareMathOperator{\dime}{dim_E}

\DeclareMathOperator{\lyap}{Lyap}

\begin{document}
\subjclass[2010]{37B30, 58E05, 37C10}
\keywords{Conley index, $\ls$-flows, (local) Morse homology, Morse-Conley-Floer homology.}
\begin{sloppypar}
\let\phi\varphi

\author[Izydorek]{Marek Izydorek\textsuperscript{1}}
\address{\textsuperscript{1} Gdansk University of Technology, Faculty of Applied Physics
and Mathematics, Gabriela Narutowicza 11/12, 80-233 Gdansk, Poland}
\author[Rot]{Thomas O. Rot\textsuperscript{2}}
\address{\textsuperscript{2}Mathematisches Institut, Universit\"at zu K\"oln, Weyertal 86 - 90, 50931 K\"oln, Germany. thomas.rot@uni-koeln.de}
\author[Starostka]{Maciej Starostka\textsuperscript{1,3}}
\address{\textsuperscript{3}Institute of Mathematics, Polish Academy of Sciences, ul. \mbox{Sniadeckich 8}, 00-656 Warsaw, Poland}
\author[Styborski]{Marcin Styborski\textsuperscript{1}}
\author[Vandervorst]{Robert C.A.M. Vandervorst\textsuperscript{4}}
\address{\textsuperscript{4}Vrije Universiteit Amsterdam. Department of Mathematics, De Boelelaan 1081a, 1081 HV Amsterdam, The Netherlands. vdvorst@few.vu.nl}
\title[Homotopy invariance]{Homotopy invariance of the Conley index and local Morse homology in Hilbert spaces}

\maketitle

\begin{abstract}
In this paper we introduce a new compactness condition --- Property-\propC ---  for flows in (not necessary locally compact) metric spaces. For such flows a Conley type theory can be developed. For example (regular) index pairs always exist for Property-\propC flows and a Conley index can be defined. An important class of flows satisfying this compactness condition are $\ls$-flows. We apply $E$-cohomology to index pairs of $\ls$-flows and obtain the $E$-cohomological Conley index. We formulate a continuation principle for the $E$-cohomological Conley index and show that all $\ls$-flows can be continued to $\ls$-gradient flows. We show that the Morse homology of $\ls$-gradient flows computes the $E$-cohomological Conley index. We use Lyapunov functions to define the Morse-Conley-Floer cohomology in this context, and show that it is also isomorphic to the $E$-cohomological Conley index.
\end{abstract}

\section{Introduction}

Conley index theory is a powerful tool to study dynamical systems. In the finite-dimensional setting there is a  well-developed theory, cf.\ \cite{conley,mischaikowmrozek}.
In the last two decades a number of infinite-dimensional extensions of the theory have been constructed, see for instance~\cite{benci1991new,gip,rybakowski2012homotopy}.  These were applied to obtain of existence and multiplicity results in variational problems for ODE's and PDE's, cf.\ \cite{gip,izydorek200122,izydorek2002conley, maksymiuk2015cohomological}. 
In this paper we approach the infinite dimensional Conley index from a  different angle. The Conley index runs into two problems in the infinite-dimensional setting. 

The first problem is that the spaces are not locally compact, which makes many basic constructions in Conley theory more difficult. Following ideas of Benci~\cite{benci1991new} we introduce a new compactness condition for flows on metric spaces. Isolated invariant set of Property-\propC flows always admit an index pair, cf.~Lemma~\ref{lemma:existence_ip}. We define the (classical) Conley index of an isolated invariant set of a Property-(C) flow as the ordinary singular homology of any index pair of the isolated invariant set. 

The second problem is that flows in infinite dimensions may be \emph{strongly indefinite}: isolated invariant sets have both an infinite number of unstable directions as well as an infinite number of stable directions. This implies that the classical Conley index of a strongly indefinite Property-\propC flow often fails to contain any topological information about the flow. This is closely related to the fact that the unit ball in an infinite-dimensional Hilbert space retracts to its boundary. 

In order to address the second problem we restrict to the class of $\ls$-flows, see Definition~\ref{defn:LS}. Many naturally occurring flows are $\ls$-flows, cf. the survey~\cite{izydoreksurvey}. We proceed to probe the topology of index pairs of $\ls$-flows with $E$-cohomology instead of singular homology, as this cohomology theory is better adapted the strongly indefinite nature of the flows. This cohomology theory was introduced by Abbondandolo in~\cite{abbondandolo1997new}. We review $E$-cohomology in Appendix~\ref{app:ecoh}. Based on this cohomology theory Starostka constructed in~\cite{starostka_morse} an infinite-dimensional extension of the cohomological Conley index, called the $E$-cohomological Conley index as the $E$-cohomology of an index pair. He showed that the $E$-cohomological Conley index of an isolated invariant set is well-defined, i.e.~does not depend on the chosen index pair. 

The classical Conley index is useful because of two important properties: non-triviality and continuation invariance. Continuation invariance allows us to deform a dynamical system to one in which the Conley index can be computed. The non-triviality of the Conley index of the deformed system implies that there exists non-trivial invariant sets in both the deformed as well as the undeformed dynamical system. This principle can be used to detect and localize bounded orbits such as equilibria and periodic orbits. In this paper we prove that the $E$-cohomological index satisfies both non-triviality and continuation invariance, cf.~Proposition~\ref{prop:nontriviality} and Theorem~\ref{thm:continuation}, which is fundamental for applications. 

Another approach towards the Conley index is the intrinsic definition given in \cite{rotvandervorst} in the finite-dimensional setting. The idea in this paper is to consider Lyapunov functions for isolating neighborhoods and define the Morse-Conley-Floer index as the local Morse homology of the Lyapunov functions. The  intrinsic definition coincides with the traditional topological definition of the Conley index. The methods in \cite{rotvandervorst} do not carry over directly to the infinite-dimensional setting for the same reasons as pointed out above.
However, for isolated invariant sets of gradient $\ls$-flows a version of Morse homology~\cite{abbondandolomajerhomology} is available. The local Morse homology of an isolated invariant set is defined by the equilibria in the isolated invariant set and certain connecting orbits between the equilibria. In Section~\ref{sec:localMorse} we show that the local Morse homology of a gradient $\ls$-flow computes the $E$-cohomological Conley index. 
In Section~\ref{sec:mcf} we use the local Morse homology for $\ls$-flows to give an intrinsic approach to the Conley index  for isolated invariant sets, which will be referred to as the Morse-Conley-Floer homology.  
The latter is a variation on the existence of a Lyapunov function and Morse-Conley-Floer homology is the defined as the local Morse homology
of the obtained $\ls$-gradient flow.
We show in  Section~\ref{sec:mcf} that the intrinsically defined Morse-Conley-Floer index is isomorphic to the $E$-cohomological Conley index.

\section{Preliminaries}
\subsection{Property-(C) flows} In this section we discuss a property of flows on a Hilbert space that allows us to perform the basic constructions in Conley index theory. Let $\mH$ be a real separable Hilbert space and $\phi$ a flow on $\mH$. If there is no confusion  about which flow we use, we will also write $x \cdot t$ for $\phi(t,x)$. For a subset $U\subset \mH$ we define the invariant set of $\phi$ in $U$ to be 
$$
\invariant (U, \phi) :=  \{x \in U\,|\, \phi(t,\mR) \subset U\}. 
$$ 
A closed and bounded set $U$ is an isolating neighborhood if
$$
\invariant(U,\phi) \subset \interior U.
$$ 
A set $S$ is invariant if $\invariant(S,\phi)=S$ and if for an invariant set $S$  there exists an isolating neighborhood $U$ with $S = \inv U$ then $S$ is said to be an isolated invariant set. Let us recall the notion of index pair.
\begin{definition}
Let $S$ be an isolated invariant set. We say that a closed and bounded pair $(N,L)$ is an index pair for $S$ if
\begin{itemize}
 \item $S=\inv{\cl{N \setminus L}} \subset \interior \overline{N \setminus L}$; 
 \item $L$ is positively invariant with respect to $N$; 
 \item $L$ is an exit set, i.e.~if for $x \in N$ there exists $t > 0$ such that $x \cdot t \not\in N$ then there exists $t' \in [0,t]$ such that $x \cdot [0,t'] \subset N$ and $x \cdot t' \in L$. 
\end{itemize}
\end{definition}
An index pair of an isolating neighborhood $U$ is an index pair $(N,L)$ contained in $U$ for the isolated invariant set $S=\invariant U$. We can carry out the basic constructions of Conley theory for flows satisfying the following property:
\begin{itemize}
\item[\propC]
Each sequence $\{x_n\}_{n\in\mN}\in\mH$ for which the set $\bigcup_{n \in \mN} x_n \cdot [-n,n]$ is bounded has a convergent subsequence. 
\end{itemize}
This property is closely related to the Palais-Smale property for gradient flows, cf.~Lemma~\ref{lem:g2} below. Note that isolated invariant sets are always compact for flows satisfying Property-\propC. 
Isolated invariant sets of Property-\propC flows always admit index pairs:
Following~\cite{benci1991new}, for any pair $(U,V)$ of subsets of $\mH$,  we introduce  
\[
  G^T_\f(U):=\set{x\in H\,\,|\,\, x\d[-T,T]\subset U}=\bigcap_{\abs t\leq T}\f(t,U),
\]
and
\[
\G_{\f}^T(U,V):=\set{x\in G_{\f}^T(U)\,\,|\,\,x\d[0,T]\cap V\neq\varnothing}.
\]
For $\G_{\f}^T(U,\partial U)$ we will write $\G_{\f}^T(U)$. 
 
\begin{lemma}\cite[Thm. 1.4.]{benci1991new}\label{lemma:existence_ip}
Let $U$ be an isolating neighborhood for a Property-\propC flow $\phi$. For all $T>0$ sufficiently large $(G_{\f}^T(U),\G_{\f}^T(U))$ is an index pair. 
\end{lemma}
\begin{remark}\label{remark:tau}
Recall that an index pair $(N,L)$ is said to be regular if the function $\tau_N\:N\to[0,\infty]$ given by 
$$
\tau_N(x)=\begin{cases}
\sup\set{t>0\,|\,x\d [0,t]\subset N\setminus L}\qquad& x\in N\setminus L\\0 \qquad &x\in L
\end{cases},
$$ 
is continuous. As was observed in~\cite[Remark 3.6]{starostka_morse},~\cite[Theorem 5.5.13]{chang2006methods} implies that regular index pairs always exist for Property-\propC flows. 
\end{remark}

(Regular) index pairs therefore always exist for isolated invariant sets of Property-\propC flows. Moreover the homology of an index pair for $U$ is independent of the chosen index pair, cf.~\cite[Theorem 1.5]{benci1991new}, and is called the classical Conley index of the isolated invariant set. For strongly indefinite flows the Conley index often vanishes, which is why we will probe the topology of the pair $(N,L)$ with another suitable cohomology theory. 

\begin{remark}
The definitions and proofs in this section also makes sense if the Hilbert space $\mH$ is replaced with any complete metric space. 
\end{remark}

\subsection{E-cohomological Conley index}
In this section we give a definition and properties of the $E$-cohomological Conley index. For our purposes, the most important feature of $E$-cohomology theory is that there is a large class of flows such that the flow deformation is an $E$-morphism. For the rest of the paper, we fix a bounded selfadjoint linear invertible operator $L:\mH\rightarrow \mH$ for which there exists a sequence of finite-dimensional subspaces $\{E_n\}$ satisying $\overline{\bigcup_{n \in \mN} E_n} = \mH$ and $L(E_n) = E_n$. 
We will write $P_n:\mH\rightarrow \mH$ for the orthogonal projection to $E_n$. The operator $L$ gives a splitting of $\mH$ in $E^+$ and $E^-$, closed subspaces corresponding to the positive and negative spectrum of $L$ respectively. The splitting allows us to use $E$-cohomology theory, a generalized cohomology theory with a restricted set of admissible morphisms and changed dimension axiom. In Appendix~\ref{app:ecoh} we give short primer on this cohomology theory and we refer to the original paper~\cite{abbondandolo1997new} for more details. 

\begin{definition}
\label{defn:LS}
Let $F =  L + K$ be a vector field on $\mH$. We say that $F$ is an $\ls$-vector field if it is globally Lipschitz and $K$ is completely continuous, i.e.~$K$ maps bounded sets into precompact sets. A flow is an $\ls$-flow is generated by an $\ls$-vector field. 
\end{definition}

\begin{remark}
\label{rem:lipschitz}
We assume that an $\ls$-vector field is globally Lipschitz for convenience only, because $\ls$-vector fields then generate a global flows. Since we are only interested in the flow around isolated invariant sets, it would suffice to work with locally Lipschitz vector fields and local flows. All results can be translated to this setting with minor notational inconveniences. 
\end{remark}
As in~\cite{starostka_morse}, we combine $E$-cohomology theory with the Conley index to get an $E$-cohomological Conley index. To be able to use the results about the existence of index pairs, we prove:

\begin{lemma}\label{lemma:C}
Every $\ls$-flow satisfies Property-\propC.
\end{lemma}

\begin{proof} 
Let $\{x_n\}_{n\in \mN}$ be a sequence and $R>0$ such that the set $X = \bigcup_{n \in \mN} x_n \cdot [-n,n]$ is bounded and contained in the ball $B_R(0)$. Let $P^\pm$ be the projections to $E^\pm$. If the sequence $\{x_n\}_{n\in\mN}$ has a convergent subsequence then both sequences $x^+_n:=P^+(x_n)$ and $x^-_n:=P^-(x_n)$ have convergent subsequences. 

Suppose $\{x_n^+\}_{n\in\mN}$ does not have a convergent subsequence. This implies that there exists an $N\in \mN$ and $\epsilon>0$ such that $\norm{x_k^++x_l^+}>\epsilon$ for all $k,l>N$ with $k\not=l$. Since the spectrum of $L$ is isolated from $0$, it follows that there exists $T>0$ such that 
$$
\frac{3R}{\epsilon}\norm{P^+x}\leq \norm{e^{LT}x},\quad\text{for all}\quad x\in B_{R}(0).
$$ 
An $\ls$-flow has the form $x\cdot t=e^{tL}x+B(x,t)$ for a map $B$ that maps bounded sets to precompact sets, cf.~\cite{gip}. Then, for all  $k,l>N$ with $k\not=l$ we have
\begin{align*}
3R&<\frac{3R}{\epsilon}\norm{x^+_k-x^+_l}\leq \norm{e^{TL}(x_k - x_l)}\\ &\leq \norm{x_k\cdot T} + \norm{x_l\cdot T}  + \norm{B(x_k,T) - B(x_l,T)}.
\end{align*}
It follows that $\norm{B(x_k,T) - B(x_l,T)}>R$, which contradicts the fact that $B$ maps bounded sets to precompact sets. The sequence $\{x_n^+\}_{n\in\mN}$ contains a convergent subsequence. By replacing $T$ with $-T$ in the argument above we obtain the same bound assuming $\{x_n^-\}_{n\in\mN}$ does not contain a convergent subsequence. It follows that $\{x_n\}_{n\in\mN}$ always has a convergent subsequence. 
\end{proof}

\begin{proposition}
Let $(N_1,L_1)$ and $(N_2,L_2)$ be two index pairs for an isolated invariant set $S$ of an $\ls$-flow. Then
$$H^*_E(N_1,L_1) \cong H^*_E(N_2,L_2).$$
\end{proposition}
\begin{proof}
If the index pairs are regular, then~\cite[Prop. 3.4]{starostka_morse} states that there is an isomorphism of $E$-cohomological Conley indices. Now suppose that $(N,L)$ is an index pair that is not regular. Following Salamon~\cite[Lem. 5.3, Rem. 5.4]{Salamon_css} we can find a sequence of regular index pairs $\{(N,L^m)\}$ such that $L^m \subset L^n$ if $n \leq m$ and $\bigcap_{m \in \mN}L^m = L$. By the continuity property of $E$-cohomology, see Lemma~\ref{lemma:continuity}, we conclude that $H^*_E(N,L) \cong H^*_E(N,L^m)$. Therefore all index pairs of a given isolated invariant set have isomorphic $E$-cohomology groups.
\end{proof} 

\begin{definition}
Let $U$ be an isolating neighborhood for an $\ls$-flow $\phi$ with $S = \invariant (U,\phi)$. The $E$-\textit{cohomological Conley index} is defined as
$$\ch(U,\phi) := H^*_E(N,L).$$
for any index pair $(N,L)$ of $S$ contained in $U$. This does not depend on the chosen isolating neighborhood $U$ of $S$ and we will define the Conley index of an isolated invariant set to be $\ch(S,\phi):=\ch(U,\phi)$ for any isolating neighborhood. We omit $\phi$ from the notation if no confusion can arise.
\end{definition}

A non-zero $E$-cohomological Conley index detects isolated invariant sets.
\begin{proposition}[Non-triviality]
\label{prop:nontriviality}
Let $U$ be an isolating neighborhood of an $\ls$-flow $\phi$. If $\ch(U,\phi) \neq 0$ then $S:=\invariant(U,\phi) \neq \varnothing$.
\end{proposition}
\begin{proof}
If $S=\varnothing$ then $N=L=U$ is an index pair in $U$ and $H_E(U,U)=0$.
\end{proof}

\begin{definition}
Two $\ls$-flows $\phi_0$ and $\phi_1$ are said to be related by continuation in $U$ if there exists a continuous family $H: [0,1] \times \mH\times\mR\to \mathbb{H}$ with
$$H(0,x,t) = \phi_0(x,t) \quad \text{and} \quad H(1,x,t) = \phi_1(x,t) \quad \text{for all}\quad  (x,t)\in \mH\times \mR,$$ 
such that each $H(s,\cdot,\cdot)$ is an $\ls$ flow with isolating neighborhood $U$.
\end{definition}

In Section~\ref{sec:continuation} we prove the continuation principle for the $E$-cohomological Conley index of $\ls$-flows. 
\begin{theorem}[Continuation principle]
\label{thm:continuation}
If two $\ls$-flows $\phi_0$ and $\phi_1$ are related by continuation in $U$ then
 $$\ch(U,\phi_0) \cong \ch(U,\phi_1).$$

\end{theorem}

\section{Proof of the Continuation Principle}
\label{sec:continuation}

Throughout the current section we will assume that the given flows are $\ls$-flows, hence can be written in the form form 
$$\f(t,x)=e^{tL}x+U(t,x).$$
To prove the continuation principle we will make use of maps induced by $\ls$-flows of the form 
$$
\Psi(x):=\phi(\tau(x),x)=e^{\tau(x)L}x+U(\tau(x),x),
$$
for a continuous function $\tau:X\rightarrow \mR$ with compact image. Such maps induce maps in $E$-cohomology, see Remark~\ref{remark:extend}.

\begin{lemma}\label{lemma:V}
  Suppose that $(N_1,L_1)$ is a regular index pair for $\f$ with the exit time function $\tau_{N_1}$ and let $(N_2,L_2)$ be another index pair such that $N_2\ss \cl{N_1\setminus L_1}$ and $\inv{\cl{N_1\setminus L_1}}=\inv{\cl{N_2\setminus L_2}}$. Assume that there is a $T_0>0$ such that $x\d[0,T_0]\cap L_1\neq\varnothing$ for all $x\in L_2$ (cf.~Figure \ref{fig:assumptionindex}). Set $Q:=(L_2\d\R_{\geq0}\cap N_1)\cup L_1$. Then:
\begin{enumerate}[(i)]
  \item The inclusion $(N_2\cup Q,L_1)\hookrightarrow (N_2\cup Q,Q)$ induces an isomorphism. 
  \item The inclusion $i\:(N_2,L_2)\hookrightarrow(N_2\cup Q,Q)$ induces an isomorphism;
  \item There is $T'>0$ such that $x\d[0,T']\cap (N_2\cup Q)\neq\varnothing$ for all $x\in N_1$;
  \item The inclusion $j\:(N_2\cup Q,L_1)\hookrightarrow(N_1,L_1)$ induces an isomorphism.
  \item The map $g\:(N_2,L_2)\to(N_1,L_1)$ defined by
\[
g(x)=\left\{
         \begin{array}{ll}
           x\d T_0 & \hbox{if $x\d[0,T_0]\ss N_1\setminus L_1$} \\
           x\d \tau(x) & \hbox{otherwise}
         \end{array},
       \right.
\]
induces an isomorphism.
\end{enumerate}
\end{lemma}

\begin{proof}
\noindent (i) Consider $\Phi\:[0,1]\times(N_2\cup Q,Q)\to(N_2\cup Q,Q)$ defined by
  \[
  \Phi(\lambda,x)=\left\{
         \begin{array}{ll}
           x\d (\lambda T_0) & \hbox{if $x\d[0,\lambda T_0]\ss N_1\setminus L_1$} \\
           x\d \tau(x) & \hbox{otherwise}
         \end{array},
       \right.
\]
The map $\Phi$ fulfills all the assumptions of Lemma~\ref{lemma:deformation_retract}. The only non-trivial thing to check is that $\Phi(\lambda,Q)\ss Q$. To see this, let $y\in Q$. If $y\in L_1$, then $\Phi(\lambda,y)=y$. Suppose that $y=x\d t'$, $x\in L_2$ and $t'\geq0$. Then $\Phi(\lambda,y)=x\d(t'+\lambda T_0)\in Q$ or $\Phi(\lambda,y)=x\d(t'+\tau(y))\in Q$.

 \noindent (ii) Since  $(N_2,L_2)=(N_2,N_2\cap Q)\hookrightarrow (N_2\cup Q,Q)$ the strong excision axiom\footnote{Recall the strong excision axiom from~\cite{abbondandolo1997new}. If $X$ and $Y$ are closed $E$-locally compact subsets of $H$ and $i:(X,X\cap Y)\hookrightarrow(X\cup Y,Y)$ is the inclusion map then $H^*_E(i)$ is an isomorphism.} implies the assertion.

 \noindent (iii)  Suppose that on the contrary, that there is a sequence $\{x_n\}_{n\in\mN}\subset N_1$ such that $x_n\d[0,2n]\cap (N_2\cup Q)=\varnothing$ for all $n\in\N$. Set $y_n:=x_n\d n$. Then $y_n\d[-n,n]\ss\cl{N_1\setminus L_1}$ and by Property \propC the sequence $\{y_n\}_{n\in\mN}$ converges up to subsequence to some $y_0\in \inv U\ss\inte{N_2\setminus L_2}$. That is, for $n$ sufficiently large $y_n=x_n\d n\in N_2$. This is a contradiction. 

 \noindent (iv) Consider $\Phi'\:[0,1]\times(N_1,L_1)\to(N_1,L_1)$ defined by
  \[
  \Phi'(\lambda,x)=\left\{
         \begin{array}{ll}
           x\d (\lambda T') & \hbox{if $x\d[0,\lambda T']\ss N_1\setminus L_1$} \\
           x\d \tau(x) & \hbox{otherwise}
         \end{array},
       \right.
\]
  where $T'>0$ is given by Part (iii). Since $N_2\cup Q$ is positively invariant with respect to $\f$ we have $\Phi'(\lambda,x)\in N_2\cup Q$ for all $x\in N_2\cup Q$. By (iii) $\Phi'(1,N_1)\subset N_2\cup Q$ and obviously $\Phi'(0,x)=x$. Thus, $\Phi$ satisfies assumptions of Lemma~\ref{lemma:deformation_retract_2}.

  \noindent (v) The conclusion follows, since $g$ is the composition of previously defined maps $g=j\circ\Phi(1,\,\d\,)\circ i$.
\end{proof}

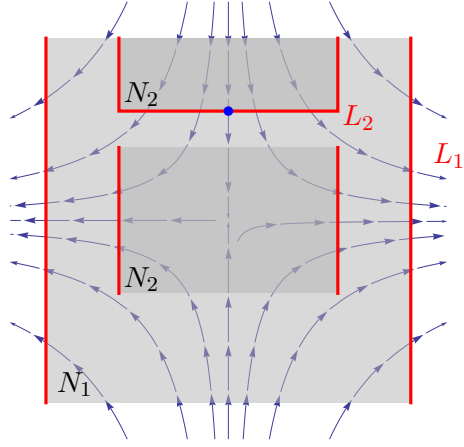
\begin{figure}
\begin{overpic}[width=0.5\textwidth]{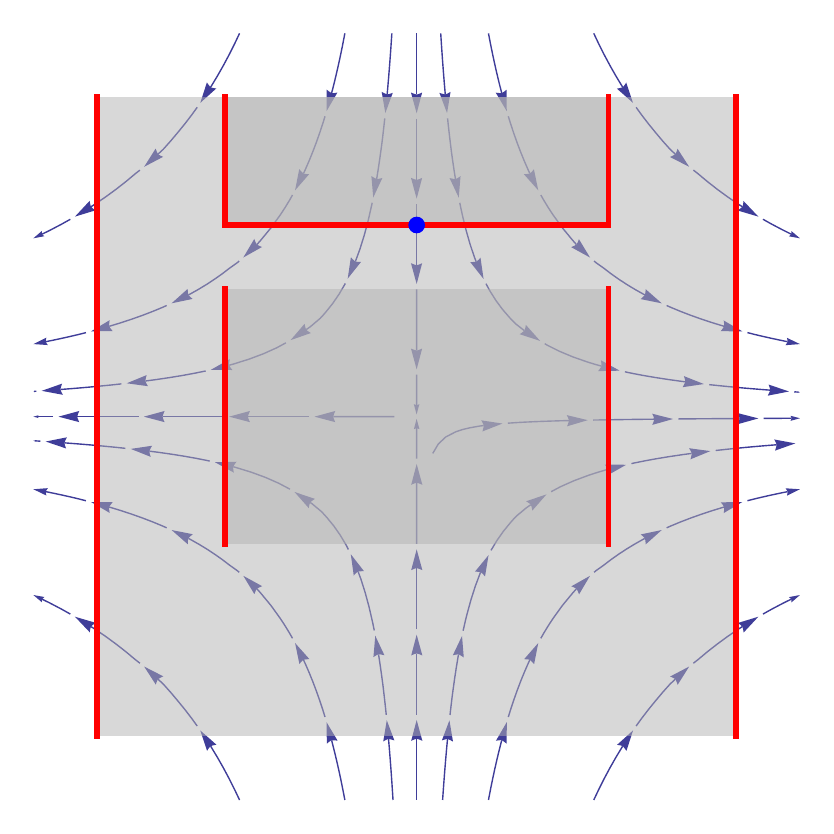}
 \put (14,14) {$N_1$}

 \put (93,62) {\textcolor{red}{$L_1$}}
 \put (28,36) {$N_2$}
 \put (28,75) {$N_2$}
 \put (74,70) {\textcolor{red}{$L_2$}}

\end{overpic}
\caption{It is not true that for nested index pairs as in Lemma~\ref{lemma:V} that there always exists a $T_0$ such that $x\cdot[0,T_0]\cap L_1\not=\emptyset$ for all $x\in L_2$, see also Remark~\ref{remark:googpairs}.}
\label{fig:assumptionindex}
\end{figure}

Let $U_\rho=N_\rho(U)=\set{x\in H\,\,|\,\,d(x,U)<\rho}$ denote the $\rho$-neighborhood of a set $U$.

\begin{lemma}\label{lemma:rho}
  Let $\f\:\R\times \mH\to \mH$ be a Property-\propC flow and $U$ be an isolating neighborhood for $\f$. Then there exist
  \begin{enumerate}[(i)]
   \item $T_0>0$ and $\epsilon_0>0$ such that $N_{\epsilon_0}(G^{T_0}_\f(U))\subset\inte U$. 
   \item $\r>0$ such that $U_\r$ is an isolating neighborhood and $\inv U=\inv{U_\r}$.
   \item $T>T_0$ such that $G_{\f}^T(U_\r)\ss\inte U$.
  \end{enumerate}
\end{lemma}

\begin{proof}
  (i) It is enough to take $\epsilon_0=\frac12(\partial U,\inv U)$. See~\cite[Lem. 5.5.19]{chang2006methods} for details.

  \noindent (ii) Suppose that on the contrary there exists a sequence of points $\{x_n\}_{n\in\mN}\subset \inv{U_{\frac1n}}$ such that $x_n\not\in\inv U$. Since $\f$ satisfies Property \propC, the sequence $\{x_n\}_{n\in\mN}$ tends up to subsequence, to some $x_0\in\inv U\subset\inte U$, which is a contradiction.

  \noindent (iii) Suppose that for all $n\geq1$ $G_{\f}^n(U_\r)\not\ss\inte U$. There is a sequence $x_n\in G_{\f}^n(U_\r)\setminus\inte U$ which is subconvergent to $x_0\in\inv{U_\r}=\inv U$. Impossible. 
\end{proof}

\noindent\textbf{Assumption:} From now on we assume that $\f,\psi\:\R\times \mH\to\ \mH$ are given $\ls$-flows, $U\subset \mH$ is an isolating neighborhood for $\f$ and 
\begin{enumerate}[({A}1)]
\item \label{enum:a2} the mapping $x\mapsto\f(t,x)$ is uniformly continuous on $U_\r$ (uniformly with respect to $t$ for all $\abs t\leq 4T$), that is for any $\epsilon>0$  there is $\delta>0$ such that for all $x,y\in U_\r$ and $\abs t\leq 4T$ the inequality $\norm{x-y}<\delta$ implies that
  \begin{equation*}
   \norm{\f(t,x)-\f(t,y)}<\epsilon.
  \end{equation*}
  Here the time $T$ is given by Lemma~\ref{lemma:rho}.
  \item \label{enum:a0}  $\rho<\min\set{\epsilon_0,2\delta(\epsilon_0)}$, where $\rho$ and $\epsilon_0$ are given by Lemma~\ref{lemma:rho} and $\delta$ is given by (A\ref*{enum:a2}).
\item \label{enum:a1}  for any $t\in[-4T,4T]$ and any $x\in U_{\rho}$
\begin{equation*}\label{eq:close}
  \norm{\f(t,x)-\psi(t,x)}<\frac\r2,
\end{equation*}
\end{enumerate}

\begin{lemma}\label{lemma:common_nbhd}
  Let $U$ be an isolating neighborhood for $\f$. Under assumption (A\ref*{enum:a1}) the set $U$ is also an isolating neighborhood for $\psi$. 
\end{lemma}

\begin{proof}
  This is a consequence of the inclusions 
  \[
    \inv{U,\psi}\ss \inv{U_{\frac\r2},\psi}\ss\tG^T(U_{\frac\r2})\ss G_\f^T(U_\r)\ss\inte U,
  \]
  where the last inclusion follows from Lemma~\ref{lemma:rho}.
\end{proof}

\begin{lemma}\label{lemmma:sequence}
  The following inclusions hold
  \[
    \tG^{4T}(U)\ss G_{\f}^{3T}(U)\ss\tG^{2T}(U)\ss G_{\f}^T(U).
  \]
\end{lemma}

\begin{proof}
  By Lemma~\ref{lemma:rho} (ii)
  \begin{equation}\label{eq:subset}
    G_{\f}^{2T}(U_\r)=G_{\f}^T(G_{\f}^T(U_\r))\subset G_{\f}^T(\inte U)\ss G_{\f}^T(U).
  \end{equation}
  We start from the inclusion on the right hand side. Using (A\ref*{enum:a1}) and \eqref{eq:subset} one has
  \[
    \tG^{2T}(U)\ss G_{\f}^{2T}(U_{\frac\r2})\ss G_{\f}^{2T}(U_\r)\ss G_{\f}^T(U).
  \]
  The same argument can be implied to the inclusions
  \[
    G_{\f}^{3T}(U)\ss\tG^{3T}(U_{\frac\r2})=\tG^{2T}(\tG^T(U_{\frac\r2}))\ss\tG^{2T}(G_{\f}^T(U_\r))\ss\tG^{2T}(U),
  \]
  and
  \[
    \tG^{4T}(U)\ss G_{\f}^{4T}(U_{\frac\r2})=G_{\f}^{3T}(G_{\f}^T(U_{\frac\r2}))\ss G_{\f}^{3T}(G_{\f}^T(U_\r))\ss G_{\f}^{3T}(U).
  \]
\end{proof}

Let $(N,L)$ be a regular index pair for $\f$. Then $U:=\cl{N\setminus L}$ is an isolating neighborhood for this flow and $(U,U\cap L)$ is a regular index pair. It is not true that $(U,U\cap L)$ is an index pair for $\psi$ in general, nor it is a regular pair. However, since $U$ is an isolating for $\psi$, Lemma~\ref{lemma:existence_ip} implies that $(\tG^{2T}(U), \Gamma_{\psi}^{2T}(U,U\cap L))$ is an index pair, but it might not be regular. In the proof of Theorem~\ref{thm:continuation} we would like to arrange things in such a way that three of four index pairs are regular. For this reason we need the following lemma which is due to Salamon, cf.~\cite[Lem. 5.3, Rem. 5.4]{Salamon_css}. 

\begin{lemma}\label{lemma:regularization}
   There is a set $L_{\beta}$ such that $(\tG^{2T}(U),L_{\beta})$ is a regular index pair for $\psi$. Moreover, for any $\beta>0$, $L_\beta$ can be chosen close to $\Gamma_{\psi}^{2T}(U,U\cap L)$ in the sense that for all $z\in L_{\beta}$ there is $w\in \Gamma_{\psi}^{2T}(U,U\cap L)$ such that $\norm{z-w}<\beta$.
\end{lemma}

\begin{remark}\label{remark:googpairs}
If we set $(N_1,L_1)=(G_{\f}^S(U),\G_{\f}^S(U,U\cap L))$ and $(N_2,L_2)=(G_{\f}^T(U),\G_{\f}^T(U,U\cap L))$ for $T>S$, then the last assumption of Lemma~\ref{lemma:V} is satisfied.  
\end{remark}

\begin{theorem}\label{thm:continuation}
  Assume that $\f,\psi$ are given $\ls$-flows satisfying (A\ref*{enum:a2}), (A\ref*{enum:a0}) and (A\ref*{enum:a1}).   Then 
\[  
  \ch(\f,U)=\ch(\psi,U).
\]
\end{theorem}

\begin{proof}
Since the index is independent of the choice of index pair we are going chose appropriate index pairs $({N_2},{L_2})$ and $({\tilde N_1},{\tilde L_1})$ for $\f$ and $\psi$ respectively such that $\HE*{N_2}{L_2}\cong\HE*{\tilde N_1}{\tilde L_1}$. Define
\begin{align*}
    (N_1,L_1)&:=(G_{\f}^T(U),\G_{\f}^T(U,U\cap L)), \\
    (\tilde N_1,\tilde L_1)&:=(\tG^{2T}(U),L_{\beta}),\\
    (N_2,L_2)&:=(G_{\f}^{3T}(U),\G_{\f}^{3T}(U,U\cap L)), \\
    (\tilde N_2,\tilde L_2)&:=(\tG^{4T}(U),\Gamma_{\psi}^{4T}(U,U\cap L)).
\end{align*}
Observe that the first three pairs are regular. It follows from Lemma~\ref{lemmma:sequence} that $N_2\ss\cl{N_1\setminus L_1}$ hence that the pairs $(N_1,L_1)$ and $(N_2,L_2)$ satisfy the assumptions of Lemma~\ref{lemma:V}. We claim that pairs $(\tilde N_1,\tilde L_1)$ and $(\tilde N_2,\tilde L_2)$ also satisfy these hypotheses. Indeed, by \cite[Lemma 5.5.19]{chang2006methods} one can find $\gamma>0$ such that $N_\gamma(\tilde N_2)\ss\tilde N_1$, provided that $T$ is sufficiently large. Now, we just take $\beta<\gamma$. The remainder of the proof is devided into 5 steps. 

\emph{Step 1.}
The map $f_1\:(\tilde N_1,\tilde L_1)\to(N_1,L_1)$ given by
 \begin{equation}
  f_1(x)=\begin{cases}
           \f(3T,x) & \hbox{if $\f([0,3T],x)\ss N_1\setminus L_1$}\\
           \f(\tau_{N_1}(x),x)& \hbox{else}
          \end{cases},
 \end{equation} 
 is well defined. Indeed, all orbits of $\f$ must leave $N_1$ through the set $L_1$ and $\tilde N_1\subset N_1$. We just need to be sure that $f_1(L_\beta)\subset L_1$. Let $y\in\G_{\psi}^{2T}$. Then $\psi(t_0,y)\in\partial U$ for some $t_0\in[0,3T]$ and  by (A\ref*{enum:a0}) and (A\ref*{enum:a1})
 \[
  d(\f(t_0,y),\partial U)\leq \norm{\f(t_0,y)-\psi(t_0,y)}+d(\psi(t_0,y),\partial U)<\frac\rho2<\frac{\epsilon_0}{2}.
 \]
 Let $x\in L_\beta$. Taking $\beta$ in Lemma~\ref{lemma:regularization} so small that $\beta<\delta(\frac{\epsilon_0}{2})$, where $\delta$ is given by (A\ref*{enum:a2}), one has a point $y\in\G_{\psi}^{2T}$ such that $\norm{x-y}<\delta(\frac{\epsilon_0}{2})$ and by the uniform continuity of $\f(t_0,\,\d\,)$ 
 \[
  d(\f(t_0,x),\partial U)\leq\norm{\f(t_0,x)-\f(t_0,y)}+d(\f(t_0,y),\partial U)<\epsilon_0.
 \]
 That is $\f([0,t_0],x)\cap L_1\neq\varnothing$ for $x\in L_\beta$.

\emph{Step 2.}
The map $\xi\:(N_2,L_2)\to(\tilde N_1,\tilde L_1)$ given by
 \begin{equation}
  \xi(x)=\begin{cases}
           \psi(3T,x) & \hbox{if $\f([0,3T],x)\ss \tilde N_1\setminus \tilde L_1$}\\
           \psi(\tau_{\tilde N_1}(x),x)& \hbox{else}
          \end{cases},
 \end{equation} 
 is well defined. 
 
 We have to show that $\xi(L_2)\subset \tilde L_1$. The reasoning goes in the same manner as above. Let $x\in L_2$. Then there is $t_1\in[0,3T]$ such that $\f(t_1,x)\in\partial U$. Using (A\ref*{enum:a0}) and (A\ref*{enum:a1}) one has
 \[
  d(\psi(t_1,x),\partial U)\leq\norm{\psi(t_1,x)-\f(t_1,x)}+d(\f(t_1,x),\partial U)<\frac{\epsilon_0}{2},
 \]
 which means that $\psi([0,t_1],x)\cap \tilde L_1\neq\varnothing$.
 
\emph{Step 3.}
By Lemma~\ref{lemma:V} Part (v) the map $g\:(N_2,L_2)\to(N_1,L_1)$ given by 
 \begin{equation}
  g(x)=\begin{cases}
           \f(3T,x) & \hbox{if $\f([0,3T],x)\ss N_2\setminus L_2$}\\
           \f(\tau_{N_1}(x),x)& \hbox{else}
          \end{cases},
 \end{equation} 
 induces the isomorphism $g^*\:\HE{*}{N_1}{L_1}\to\HE*{N_2}{L_2}$.
 
\emph{Step 4.}
Both triangles of the diagram 
 \begin{equation}\label{diagram_contin_maps}
\xymatrix{
(N_1,L_1) & \ar[l]_{g}    (N_2,L_2) \ar[d]^{\xi}  \\
                                & \ar[lu]^{f_1} (\tilde N_1,\tilde L_1) & \ar[l]^{\tilde g} (\tilde{N}_2,\tilde{L}_2) \ar[lu]_{f_2} }
  \end{equation}
  are homotopy commutative. Here $\tilde g$ and $f_2$ are defined in the same manner as in Step~1 and Step~3. We will show that $g\sim f_1\circ\xi$. The homotopy $\tilde g\sim\xi\circ f_2$ can be shown similarly. Let $h\:[0,1]\times(N_2,L_2)\to(N_1,L_1)$ be given by
  \[
    h(\lambda,x)=\f\left(\min\set{3T,\tau_{N_1}(z_\lambda(x))},z_\lambda(x)\right),
  \]
  where $z_\lambda(x)=\psi(\min\{3\lambda T,\tau_{\tilde N_1}(x)\},x)$. It is routine to verify that $h(0,\,\d\,)=g$ and $h(1,\,\d\,)=f_1\circ\xi$. The only thing to check is $h(\lambda,L_2)\ss L_1$ for all $\lambda\in[0,1]$. 
  
  In fact, it is enough to show that for all $0\leq\lambda\leq1$ and $x\in L_2$, there is $t_\lambda\in[0,3T]$ such that $\f(t_\lambda,\psi(3\lambda T,x))$ lies outside $N_1$. Observe that $\psi(3\lambda T,x)\in\tilde N_1\ss N_1$. There is $t'_\lambda\in[3\lambda T,3T]$ such that $\f(t'_\lambda,x)\in\partial U$. Let $t_\lambda=t'_\lambda-3\lambda T$. Then $\f(t_\lambda,\f(3\lambda T,x))\in\partial U$ and by uniform continuity of $\f$ we have
  \begin{align*}
    d(\f(t_\lambda,\psi(3\lambda T,x)),\partial U)&\leq\norm{\f(t_\lambda,\psi(3\lambda T,x))-\f(t_\lambda,\f(3\lambda T,x))}\\&+d(\f(t_\lambda,\f(3\lambda T,x)),\partial U)<\epsilon_0
  \end{align*}
  since $\norm{\f(3\lambda T,x)-\psi(3\lambda T,x)}<\frac\rho2<\delta(\epsilon_0)$ by (A\ref*{enum:a0}).

\emph{Step 5.}
It is a direct consequence of Step~4 that the diagram 
  \begin{equation}\label{diagram_contin_maps}
\xymatrix{
\HE*{N_1}{L_1} \ar[r]^{g^*}  \ar[dr]_{f_1^*} &  \HE*{N_2}{L_2} \ar[dr]^{f_2^*}  \\
                                & \HE*{\tilde N_1}{\tilde L_1} \ar[u]_{\xi^*}  \ar[r]_{\tilde g^*}   & \HE*{\tilde{N}_2}{\tilde{L}_2}  }
\end{equation}
commutes, i.e., $g^*=\xi^*\circ f_1^*$ and $\quad \tilde g^*=f_2^*\circ \xi^*$. Since $g^*$ and $\tilde g^*$ are isomorphisms, so is $\xi^*$, which completes the proof.

\end{proof}

\section{Reineck's Theorem for $\ls$-flows}
\label{sec:reineck}

Reineck~\cite{reineck1991} proved the following theorem: 

\begin{theorem}[\cite{reineck1991}]\label{thm:classicalReineck}
Let $X$ be a smooth vectorfield on a Riemannian manifold $M$, and let $S$ be an isolated invariant set in the flow generated by $X$, with isolating neighborhood $N$. Then $S$ can be continued to an isolated invariant set in a gradient flow without changing $X$ on $M\setminus N$.
\end{theorem}

We say that a function $b$ has finite dimensional support if there exists an $n$ such that $b(P_n(x))=x$ for all $x$. A function with finite-dimensional support has a compact gradient. Here we prove an analogue of Reineck's theorem for $\ls$-flows.

\begin{theorem}[Reineck's Theorem for $\ls$-flows]
\label{theorem:LSReineck}
Let $F = L + K$ be a continuously differentiable $\ls$-vector field and let $U$ be an isolating neighborhood for the induced flow $\phi$. Then $\phi$ is related by continuation in $U$ to a gradient flow of a function $f:U\rightarrow \mR$ of the form 
$$f(x) = \frac{1}{2} \ip{Lx}{x} - b(x),$$ 
where $b$ has finite-dimensional support.
\end{theorem}

\begin{proposition} \label{prop:close_fields}
Let $U$ be an isolating neighborhood of an $\ls$-flow $\phi_0$ generated by the $\ls$-vector field $F_0$. Then there exists an $\epsilon>0$, such that if an $\ls$-vector field $F_1$ satisfies
$$
\norm{F_0(x) - F_1(x)} < \epsilon,\qquad \text{for all}\qquad x\in U
$$
then the flows of $\phi_0$ and $\phi_1$, generated by the vector field $F_1$ are related by continuation.
\end{proposition}

\begin{proof}

Let $c$ be a Lipschitz constant for $F_0$. Fix $p \in U$ and let $x_i$, $i\in\{0,1\}$, be the solutions to $\dot{x}_i(s) = F_i(x_i(s))$ with $x_i(0) = p$.
\begin{align*}
  \norm{x_0(t) - x_1(t)}  &= \abs{\int_0^t [F_0(x_0(s))- F_1(x_1(s))] \, ds} \\
  &\leq \int_0^t \norm{F_0(x_0(s))-F_0(x_1(s))} + \norm{F_0(x_1(s))-F_1(x_1(s))} \, ds \\
  &\leq \int_0^t (c  \norm{x_0(s) - x_1(s)} + \epsilon ) \, ds.
\end{align*}
By Gronwall's inequality we get
$$\norm{x_0(t) - x_1(t)} \leq \epsilon t e^{ct} \leq \epsilon T e^{cT},$$
for all $t \in [-T,T]$.
We can choose $\rho$, $T$ and $\epsilon$ such that (compare Lemmas~\ref{lemma:rho} and~\ref{lemma:common_nbhd})
 \begin{enumerate}
 \item $\inv{\phi_0,U_{\rho}} = \inv{\phi_0,U}$;
 \item $G^T_{\phi_0}(U_{\rho}) \subset \interior U$;
 \item $\epsilon T e^{cT} < \frac{\rho}{2}$.
 \end{enumerate}
Then $U$ is also an isolating neighborhood for each $\ls$-flow $\phi_s$ induced by $F_s(x) := (1-s)F_0(x) + sF_1(x)$. To see this note that 
$$\inv{\phi_s,U} \subset G^T_{\phi_s}(U) \subset G^T_{\phi_1}(U_{\frac{\rho}{2}}) \subset G^T_{\phi_0}(U_\rho) \subset \interior U.$$
\end{proof}

Recall that we assumed the existence of $\{E_n\}_{n\in\mN}$, an increasing sequence of finite-dimensional subspaces of $\mH$ such that $L(E_n) = E_n$ and $\cl{\bigcup_{n \in \mathbb{N}} E_n} = \mH $. Denote by $P_n:\mH \to \mH$ the orthogonal projection onto $E_n$. From the above proposition we have the following corollary.

\begin{corollary}\label{cor:finite_range}
For sufficiently large $n$, the flows induced by the $\ls$-vector fields $F = L + K$ and $F_n = L + P_nK$ are related by continuation.
\end{corollary}

Pick a nonzero vector $v \in \mathbb{H}$ and put $K(x) = \norm{x}v$. Then on a unit ball $\norm{K - P_nKP_n}_{sup} = 1$. Therefore Proposition~\ref{prop:close_fields} does not show that the flows induced by $F = L + K$ and $F_n =L+P_nKP_n$ are related by continuation. This is true however and we adopt the technique from~\cite[Lemma 4.1]{gip} to show this. We will need the following result.

\begin{proposition}[{\cite[Theorem IX.3.2]{maurin1991analiza}}]\label{prop:maurin} 
Let $U$ be open subset of Banach space $X$ and $V$ an open subset of Banach space $Y$. Moreover, suppose that $f:U \times V \to X$ is continuous and is derivatiable in the $X$ direction and that the map
$$U \times V \to X \ni (x,y) \mapsto f'_X(x,y) \in L(X,X)$$
is continuous. Then $f$ induces a continuous family of local flows.
\end{proposition}

\begin{proposition} \label{prop:KP_n}
Let $F = L + K$ be a continuously differentiable $\ls$-vector field and let $U$ be an isolating neighborhood. Then for sufficiently large $n$ the flows induced by $F$ and $F_n = L + P_nKP_n$ are related by continuation in $U$.
\end{proposition}

\begin{proof}
By Corollary~\ref{cor:finite_range} we can assume without loss of generality that $P_nK = K$. Define $H: [0,1] \times U \to \mathbb{H}$ by $H(\lambda,\cdot) = L + K((1+n)(1-n\lambda)P_{n+1} + n[(n+1)\lambda - 1]P_n)$ for $\lambda \in (\frac{1}{n+1},\frac{1}{n}]$ and $H(0,\cdot)=L+K$. If $H$ induces a continuous family of local flows then, by the compactness of the isolated invariant set, there exists $s > 0$ such that  $U$ is an isolating neighborhood for $H(\lambda,\cdot)$ provided $\lambda \in [0,s)$. If this is the case, it is enough to take $n$ such that $\frac1n < s$. 

To show that $H$ indeed induces a continuous family of local flows we will check that $H$ satisfies the assumptions of Proposition~\ref{prop:maurin}. We therefore examine continuity of the map $(\lambda,x) \mapsto D_XH(\lambda,x)$. Let $(\lambda_n,x_n)$ be a sequence converging to $(\lambda,x)$. The only nontrivial case is to check that for $\lambda=0$ the sequence $D_XH(\lambda_n,x_n)$ tends to $D_XH(0,x)$. Without loss of generality we can assume  $\lambda_n = \frac{1}{n}$. We have
\begin{align*}
\Vert D_XH(\lambda_n,x_n)&- D_XH(0,x)\Vert = \norm{DK(P_n(x_n))P_n - DK(x)} \\
  &= \norm{DK(P_n(x_n))P_n - DK(x)P_n + DK(x)P_n - DK(x)} \\
  &\leq  \norm{DK(P_n(x_n))P_n - DK(x)P_n} +  \norm{DK(x)P_n - DK(x)} \\
  &\leq \norm{DK(P_n(x_n)) - DK(x)} + \norm{DK(x)P_n - DK(x)}.
\end{align*}
The sequence $P_n(x_n)$ converges to $x$ so that $\norm{DK(P_n(x_n)) - DK(x)}$ converges to $0$ by the assumption that $K$ is continuously differentiable. On the other hand 
$$\norm{DK(x)P_n - DK(x)}=\norm{P_nDK(x)^* - DK(x)^*} \to 0,$$ 
since the adjoint $DK(x)^*$ is compact.
\end{proof}

\begin{proof}[Proof of Theorem~\ref{theorem:LSReineck}]
By Proposition~\ref{prop:KP_n} there exists an $n\in\mN$ such that the flow $\f$ is related by continuation in $U$ to the product flow $(\psi,\psi^\perp)$. Here the product flow is induced by the vector field $F_n\:E_n\oplus (E_n)^\perp\to E_n\oplus (E_n)^\perp$ given by
$$F_n(x,y)=(Lx+P_nK(x),Ly).$$ 
Note that the space $E_n$ is finite-dimensional and that $U_n:=U\cap E_n$ is an isolating neighborhood for $\psi$. Then Theorem~\ref{thm:classicalReineck} states that the flow $\psi$ can be continued in $U_n$ to a gradient flow of a function defined on some neighborhood of $U_n$ in $E_n$ of the form 
$$\tilde f(x)=\frac12\ip{Lx}{x}+b(x).$$ 
It follows that the flow $(\psi,\psi^\perp)$ can be continued to the gradient flow of 
$$f(x,y)=\tilde f(x)+\frac12\ip{Ly}{y}.$$ 
\end{proof}

\section{Morse homology in Hilbert spaces}
\label{sec:localMorse}
We use the construction of Morse homology of~\cite{abbondandolomajerhomology}. These results have been extended beyond the flat Hilbert space setting~\cite{abbondandolomajercomplex}, which we will not use here. Define
$$
\mathcal{F}(L)=\{f\in C^2(E)\,|\, f(x)=\frac{1}{2}\langle Lx,x\rangle+b(x),\quad \nabla b \quad\text{is compact}\}.
$$
Here and below we always use the Hilbert metric to define the gradient flow. Assume that\footnote{In\cite{abbondandolomajerhomology} the Morse homology is defined for an action window $I$, which we do not need here. In their notation we use $I=\mR$, and write $H_q(f)$ for $H_q(f,\mR)$.}:
\begin{enumerate}
\item[(F1)] $f\in \mathcal{F}(L)$.
\item[(F2)] $f$ satisfies the Palais-Smale (PS) condition, i.e.~any sequence $\{x_n\}_{n\in\mN}\in \mH$ with $f(x_n)\rightarrow c$ and $\nabla f(x_n)\rightarrow 0$ has a convergent subsequence.
\item[(F3)] $f$ is a Morse function.
\item[(F4)] $f$ satisfies the Morse-Smale condition up to order $2$. That is, for any two critical points $x$ and $y$ with index difference\footnote{The Morse index $m_{E^-}(x)$ is defined as the relative dimension of the negative eigenspace of the Hessian at $x$ and $E^-$, see~\cite{abbondandolomajerhomology} for details.} $m_{E^-}(x)-m_{E^-}(y)\leq 2$ we have that $T_pW^u(x)+T_pW^s(y)=\mH$ for all $p\in W^u(x)\cap W^s(y)$.  
\item[(F5)] For every $c\in \mR$ and every $k\in \mZ$ the set $\mathrm{crit}_k (f)\cap f^{-1}(-\infty,c)$ is finite.
\end{enumerate}
Abbondandolo and Majer then proceed to define a $\mathbb{Z}_2$-graded Morse homology for $f$ satisfying (F1)-(F5), and give various invariance theorems. We recall one.
\begin{theorem}\cite[Theorem 1.8]{abbondandolomajerhomology}
Assume that $f_0,f_1$ satisfy (F1)-(F5). If $\norm{f_1-f_0}_\infty<\infty$ then $H_k(f_0)\cong H_k(f_1)$.
\end{theorem}

We will localize this to isolating neighborhoods $U$ of the gradient flow of $f$. To work with flows instead of local flows we assume that $\nabla f$ is globally Lipschitz, see Remark~\ref{rem:lipschitz}.

\subsection{Local Morse homology}
\label{localMH}

Morse homology is defined intrinsically by the gradient flow. The complex is the $\mZ_2$ vector space generated by the equilibria of the flow which are the critical points of the Morse function. These complex is graded by the relative Morse index. Define the moduli spaces of parametrized curves
$$
\mathcal{M}(x,y):=\{u:\mR\rightarrow \mH\,|\, \dot u(t)=-\nabla f(u(t)), \lim_{t\rightarrow-\infty}u(t)=x, \lim_{t\rightarrow \infty}u(t)=y\}\\
$$ 
There is a free $\mR$-action on these spaces by time reparametrization, i.e.  the action $s\cdot u(\cdot)\mapsto u(\cdot+s)$. Then for two critical points $x,y$ with $m_{E^-} (x)=m_{E^-} (y)+1$, the moduli space of unparametrized connecting orbits $\widehat{\mathcal{M}}(x,y):=\mathcal{M}(x,y)/\mR$ is a finite set of points. The matrix coefficient of the boundary operator is defined to be this count modulo two, i.e.
$$
\partial x=\sum\left(\# \widehat{\mathcal{M}}(x,y)\right)y,
$$
where the sum runs over all critical points $y$ with $m_{E^-}(y)=m_{E^-}(x)-1$. The differential  $\partial$ counts the number of connecting orbits between critical points of index difference one. The fundamental relation $\partial^2=0$ is proved in Morse homology using compactness and gluing arguments. For two critical points $x\in \mathrm{crit}_k(f)$, $z\in \mathrm{crit}_{k-2}(f)$, one shows that the moduli spaces $\widehat{\mathcal{M}}(x,z)$ of unparameterized solutions to the gradient flow can be compactified to a one-dimensional manifold with boundary by adjoining broken orbits. The boundary components are exactly the once broken orbits. As the number of boundary components of a one-dimensional manifold is even, which is what $\partial^2$ counts, one concludes that $\partial^2=0$ and the Morse homology is defined.

We localize the construction of Abbondandolo and Majer to isolating neighborhoods, see also~\cite{rotvandervorst,Rot:ww,Rot:2014ku} for discussions in the finite-dimensional situation. Let $U$ be an isolating neighborhood of the gradient flow of $f$. As the gradient flow of $f$ satisfying (F1) is an $\ls$-flow it in particular satisfies Property (C). This implies that the isolated invariant set $S$ is compact and that the distance $d(\inv U, \partial U)>0$. Then we define the local moduli space of parametrized orbits in $U$ by
\begin{align*}
\mathcal{M}(x,y;U)&:=\{u\in \mathcal{M}(x,y)\,|\, u(t)\in U \quad \text{for all} \quad t\in \mR\}.
\end{align*}
There are evaluation maps $\ev:\mathcal{M}(x,y)\rightarrow \mH$ defined by $ev(u)=u(0)$. These are embeddings with image $W^{u}(x)\cap W^s(y)$, cf.~\cite[Corollary 4.5]{abbondandolomajerhomology}. By isolation, the set $\ev(\mathcal{M}(x,y;U))$ is given by the union of the components of $\ev(\mathcal{M}(x,y))$ that are completely contained in $U$. But as $\ev(\mathcal{M}(x,y;U))\subset S$ this implies that the closure $\mathrm{cl}(\ev(\mathcal{M}(x,y;U)))\subset S\subset \mathrm{int}(U)$. Thus the broken orbits used to define the compactification of the moduli space of unparameterized orbits have image completely in $\mathrm{int}(U)$. Moreover if there exists a pair of unparameterized orbits $([u],[v])\in \mathcal{M}(x,y;U)/\mR\times \mathcal{M} (y,z;U)/\mR$ with $m_{E^-}(x)=m_{E^-} (y)+1=m_{E^-}(z)+2$, then the proof of~\cite[Proposition 6.2]{abbondandolomajerhomology} shows that there is an essentially unique family of unparameterized orbits in $\mathcal{M}(x,y)/\mR$ converging to the broken orbit. But the image of this family must eventually lie in $S$ and hence in the interior of $U$ by the compactness Property-\propC as $d(S,\partial U)>0$. We have argued that the differential $\partial(f,U)$, obtained by counting orbits of index difference one \emph{completely} contained in $U$ squares to zero. We define the local Morse homology as $H_k(f,U):=\ker \partial_k(f,U)/\mathrm{im}\, \partial_{k+1}(f,U)$. In the next section we investigate the invariance property of this local Morse homology. It is clear that we need less than (F1)-(F5) to define the local Morse homology as we only need information of the flow on isolating neighborhoods. To be precise to define the local Morse homology for $(f,U)$ the following two assumptions are enough.
\begin{enumerate}
\item[(B1)] $U$ is an isolating neighborhood of the gradient flow of a $f\in C^2(\mH,\mR)$, which has the form 
$$
f(x)=\frac{1}{2}\langle Lx,x\rangle+b(x),\qquad \text{for all}\qquad x\in U,
$$
where $\nabla b(x)$ is a compact operator. 
\item[(B2)] $f$ is a Morse function on $U$ and Morse-Smale up to order $2$ on $U$. That is, all critical points in $U$ are non-degenerate and for every $p\in ev(\mathcal{M}(x,y;U))$ with $m_{E^-}(x)+m_{E^-}(y)\leq 2$ we have that $T_pW^u(x)+T_p(W^s(y)=\mH$. 
\end{enumerate}

\begin{theorem}
\label{thm:localmorse}
The local Morse homology of a pair $(f,U)$ satisfying (B1) and (B2) above is well defined. 
\end{theorem}

Before we prove this theorem, we give a Lemma that shows that Property-\propC is closely related to the Palais-Smale condition. 
\begin{lemma}
\label{lem:g2}
Let $f \in C^{1,1}(\mH,\mR)$ and suppose that the gradient flow $\phi$ satisfies Property-\propC. Let $\{x_n\}_{n\in\mN}\subset \mH$ be a bounded sequence such that $F(x_n) := \nabla f(x_n) \rightarrow 0$. Then $\{x_n\}_{n\in\mN}$ has a convergent subsequence. 
\end{lemma}

\begin{proof}
Let $c$ be a Lipschitz constant of $F$ and let $R>0$ such that $\norm{x_n} < R$ for every $n$. By passing along a subsequence, we may assume that
$$\norm{F(x_n)} \leq \frac{e^{-cn}}{n},\qquad \text{for all}\quad n\in\mN.$$
Fix $n\in \mN$ and define $\gamma_n:[0,n]\rightarrow \mR$ by $\gamma_n(t)=\norm{\phi(x_n,t)-\phi(x_n,0)}$. Then
\begin{align*}
\gamma_n(t) &\leq \int_0^t \norm{F(\phi(x_n,s))} \, ds \\
&\leq\int_0^t \norm{F(\phi(x_n,s)) - F(\phi(x_n,0))} + \norm{F(\phi(x_n,0))} \, ds\\ &\leq c  \int_0^t \gamma_n(s)\, ds + e^{-cn}.
\end{align*}
By Gronwall's inequality we have
$$\gamma_n(t) \leq e^{-cn}e^{ct} \leq 1$$
for $t \in [0,n]$ and therefore $\phi(x_n,[0,n])$ lies inside the ball of radius $R+1$ for every $n$. Analogously, we show that $\phi(x_n,[-n,0])$ is bounded. By Property-\propC the sequence $x_n$ converges up to a subsequence.
\end{proof}

\begin{proof}[Proof of Theorem~\ref{thm:localmorse}]
Let $U$ be an isolating neighborhood of the gradient flow of $f\in C^2(\mH,\mR)$. The assumptions (G1)-(G5) below suffice for defining the local Morse homology without any change in the proofs in \cite{abbondandolomajerhomology}.
\begin{enumerate}
\item[(G1)] The function $f$ has the form
$$
f(x)=\frac{1}{2}\langle Lx,x\rangle+b(x),
$$
with $\nabla b(x)$ compact, for all $x\in U$.
\item[(G2)] $f$ satisfies (PS) on $U$. That is, every sequence $\{x_n\}_{n\in\mN}\subset U$, with $f(x_n)\rightarrow c$, and $\nabla f(x_n)\rightarrow 0$, has a convergent subsequence $x_n\rightarrow x\in U$. 
\item[(G3)] $f$ is a Morse function on $U$.
\item[(G4)] $f$ satisfies Morse-Smale on $U$ up to order 2 in $U$. That is, for every $p\in \ev(\mathcal{M}(x,y;U))$ with $m_{E^-}(x)+m_{E^-}(y)\leq 2$ we have that $T_pW^u(x)+T_pW^s(y)=\mH$.
\item[(G5)] For every $c$ and every $k\in \mZ$ the set $\mathrm{crit}_k (f)\cap(-\infty,c)\cap U$ is finite.
\end{enumerate}
There is a redundancy in these axioms. Lemma~\ref{lem:g2} states that (G2) is always satisfied by functions satisfying (G1). Property (G5) is always satisfied by functions satisfying (G1) and (G3). To see this, suppose (G1) and (G3) hold, and let $\{x_n\}_{n\in\mN}$ be a sequence enumerating critical points in $U$. This sequence has a convergent subsequence by Lemma~\ref{lem:g2} and this limit must be a critical point in $U$. The function $f$ is Morse, which implies that critical points are isolated. It follows that there exist only a finite number of critical points. The assumptions (B1) and (B2) are equivalent to (G1), (G3) and (G4), but we have show that then (G2) and (G5) also hold,  and hence the local Morse homology is well defined. 
\end{proof}

\subsection{Morse-Smale functions are dense}
Recall that we assume that there exists a sequence $\{E_n\}_{n\in\mN}\subset \mH$ of finite-dimensional spaces with orthogonal projections $P_n:\mH\rightarrow \mH$ such that $\overline{\bigcup_{n\in\mN} E_n}=\mH$. Also recall that a function $b$ for which there exists an $n$ such that $b(x)=b(P_n(x))$ for all $x$ is said to have finite-dimensional support. A function with finite-dimensional support has compact gradient. We have the following density result.
\begin{proposition}
\label{prop:dense1}
Let $f=\frac{1}{2}\langle Lx,x\rangle+b(x)$ and $U$ an isolating neighborhood of the gradient flow. Then for every $\epsilon>0$ there exists an $f'=\frac{1}{2}\langle Lx,x\rangle+b'(x)$ with $\sup_{x\in U} |f(x)-f'(x)|<\epsilon$ such that $f'$ satisfies (B1) and (B2) on $U$ and $b'$ has finite-dimensional support. The gradient flows of $f$ and $f'$ are related by continuation through gradient flows of functions satisfying (B1). 
\end{proposition}
\begin{proof}
The gradient of $f_n(x):=\frac12\langle Lx,x\rangle +b(P_nx)$ equals
$$
\nabla f_n=Lx+P_n\nabla b P_n.
$$
The gradient flow of $f_n$ is related by continuation to the gradient flow of $f$ for $n$ sufficiently large by Proposition~\ref{prop:KP_n}. Note that the proof of Proposition~\ref{prop:KP_n} actually shows that the homotopy is through gradient flows of functions satisfying (B1). Now, as $U$ is bounded, it is contained in a ball $B_R(0)$. By compactness of $\nabla b$ there exists for every $\epsilon>0$ an $n$ such that $\nabla b(B_R(0))\subset P_n(\mH)+B_\epsilon(0)$. From this we estimate for $x\in B_R(0)$ that
\begin{align*}
|b(x)-b(P_n(x))|&=|\int_0^1\langle \nabla b(P_n x+s P_n^\perp x),P_n^\perp(x)\rangle ds|\\
&\leq\int_0^1|\langle P_n^\perp \nabla b(P_nx+s P_n^\perp x),x\rangle|ds\leq \epsilon \norm x\leq \epsilon R.
\end{align*}
Hence we get that $|b-b\circ P_n|_\infty\rightarrow 0$ as $n\rightarrow \infty$ on $U$. We have shown that the space of functions $f'=\frac{1}{2}\langle Lx,x\rangle+b'(x)$ satisfying (B1) on $U$, where $b'$ has finite-dimensional support, is dense in the space of all functions satisfying (B1). Such a function satisfies (B2) if and only if the restriction to $P_n(\mH)$ is Morse-Smale on $U\cap P_n(\mH)$. But in the finite-dimensional case it is a classic fact, cf.~\cite[Remark 6.7]{banyaga}, that functions whose gradient flow is Morse-Smale functions are dense. Density is transitive, which proves the proposition.
\end{proof}

\subsection{Invariance of Local Morse homology}

\begin{proposition}
\label{prop:continuation1}
Let $b_\lambda$ be a smooth family of functions with compact gradient, and $U$ an isolating neighborhood of each gradient flow of $f_\lambda(x):=\frac{1}{2}\langle Lx,x\rangle+b_{\lambda}(x)$. Assume that $f_0,f_1$ satisfy additionally (L2). Assume that $\sup_{x\in U}\norm{\frac{\partial f_\lambda}{\partial \lambda}(x)}<C$. Then $H_*(f_0,U)\cong H_*(f_1,U)$.
\end{proposition}
\begin{proof}
Let $\mH':=\mH\oplus \mR$ and equip $\mH'$ with inner product 
$$
\langle (x,\mu),(y,\lambda)\rangle_{\mH'}=\langle x,y\rangle_\mH+\frac{\mu\lambda}{\kappa},
$$  
for some $\kappa>0$ to be specified later. Define $L':\mH'\rightarrow \mH'$ by $L':=L\oplus\mathrm{Id}$. Let $\omega:\mR\rightarrow \mR$ be a smooth cutoff function with $\omega(t)=1$ for $t\leq \frac{1}{3}$ and $\omega(t)=0$ for $t\geq \frac{2}{3}$ which is strictly decreasing on $(\frac{1}{3},\frac{2}{3})$. Let $\eta$ be a smooth cutoff function with 
$$
\begin{array}{cccccc}
\eta(\mu)=0 &\text{for}&\mu\in[0,1],&\eta'(\mu)<0&\text{for} &\mu<0,\\
\eta'(\mu)>0&\text{for}& \mu>0,&|\eta'(\mu)|>1&\text{for}&\mu\in\mR\setminus[-\frac 12,\frac 32]
\end{array}
$$ 

Define $F:\mH'\rightarrow \mR$ by
$$
F(x,\mu)=f_{\omega(\mu)}+r(1+\cos(\pi \mu)+\frac{1}{\kappa}\eta(\mu))
$$
The negative gradient of $F$ equals 
\begin{align}
\label{eq:gradF}
\begin{split}
-\nabla F(x,\mu)=&-\nabla f_{\omega(\mu)}(x)\\
&-\kappa\left(\omega'(\mu)\frac{\partial f_\lambda}{\partial \lambda}\bigr|_{\lambda=\omega(\mu)}-r\pi \sin(\pi \mu)+\frac{\eta'(\mu)}{\kappa}\right)\frac{\partial}{\partial\mu}.
\end{split}
\end{align}
Let $U':=U\times[-\frac{1}{3},\frac{4}{3}]$. If $r>\frac{2\max_{\mu\in[0,1]}|\omega'(\mu)C|}{\sqrt{3}\pi}$ the critical points of $F$ in $U'$ are at $\mu=0$ and $\mu=1$. Note that the vector field in Equation~\eqref{eq:gradF} is well-defined for $\kappa=0$ and that $U'$ is an isolating neighborhood for this vector field. By Proposition~\ref{prop:close_fields} it follows that for $\kappa>0$ small, the set $U'$ is an isolating neighborhood of $F$. We analyze the local Morse homology of $F$. The critical points of $F$ can be identified with the critical points of $f_0$ at $\mu=0$ and with the critical points of $f_1$ at $\mu=1$. The index is shifted by one at $\mu=0$ as there is one extra unstable direction. So we have the identification
\begin{equation}
\label{eq:identification}
C_k(F,U')\cong C_{k-1}(f_0,U)\oplus C_{k}(f_1,U).
\end{equation}
By~\cite[Theorem 1.14]{abbondandolomajerhomology} there exists a boundary operator $\partial(F,U')$ such that $\partial(F,U')$ counts connecting orbits whenever the orbits are transverse. A moments inspection shows that the differential has the form
$$
\partial(F,U')=\left(\begin{array}{cc}\partial(f_0,U)&0\\ \Phi_k^{10}&\partial(f_1,U)\end{array}\right),
$$
with regards to the identification in \eqref{eq:identification}. As the vector field always has a positive component in $\frac{\partial}{\partial \mu}$ direction for $\mu\in(0,1)$ there are no connections from the critical points at $\mu=1$ to the critical points at $\mu=0$. From $\delta(F,U')^2=0$ we see that $\Phi^{10}_k$, associated with counting connecting orbits from the complex at $\mu$ to the complex at $\mu=1$, is a chain map, hence induces a map $\Phi^{10}_k:H_*(f_0,U)\rightarrow H_*(f_1,U)$. Iterating this trick shows that $\Phi^{10}_k$ is functorial on the homology level, independent of the chosen isolating homotopy, and $\Phi^{00}=\mathrm{Id}$, cf.~\cite{weber,rotvandervorst} for the arguments in the finite-dimensional situation. It follows that $H_*(f_0,U)\cong H_*(f_1,U)$.
\end{proof}

\begin{lemma}
\label{lem:bounded}
Suppose $f\in \mathcal{F}(L)$. Then $f$ is bounded on bounded sets.
\end{lemma}
\begin{proof}
Let $x_0,x_1\in B_R(0)$, then using that $\nabla b$ is a compact map, we obtain
\begin{align*}
|b(x_1)-b(x_0)|&=|\int_0^1\langle \nabla b(x_0+s (x_1-x_0),x_1-x_0\rangle ds|\\
&\leq \norm{ \nabla b(B_R(0))} 2 R\leq C,
\end{align*}
for some constant $C$ which only depends on $r$. As $L$ is a bounded linear operator we get that $f$ is bounded on bounded sets.
\end{proof}

\begin{corollary}\label{cor:morse_invariance}
Let $f_\lambda$ with $\lambda\in [0,1]$ be a family of functions, and $U$ an isolating neighborhood of the each gradient flow of $f_\lambda$. Assume that $f_\lambda$ satisfies (B1) for each $\lambda\in [0,1]$ and that $f_0,f_1$ satisfy additionally (B2). Then $H_*(f_0,U)\cong H_*(f_1,U)$.
\end{corollary}
\begin{proof}
For every $\epsilon>0$, there exists a smooth family $f_\lambda'$ and $N\in \mN$ with $f_0'=f_0$, $f_1'=f_1$ and
$$
f_\lambda'=(i+1-N\lambda )f'_{\frac{i}{N}}+(N\lambda-i)f'_{\frac{i+1}{N}}\qquad \text{if}\qquad \lambda\in[\frac{i}{N},\frac{i+1}{N}]
$$
Moreover we can assume $f_{\frac{i}{N}}'$ satisfies (B1) and (B2) and that the family $s\mapsto (1-s)f_\lambda+s f'_\lambda$ induces an continuation of negative gradient flows for each $\lambda$. We use the fact here that $[0,1]$ is compact, isolation is open and that the Morse-Smale functions are dense in $U$, cf.~Propositions~\ref{prop:close_fields} and~\ref{prop:dense1}. This family is a piecewise linear approximation of the continuation $f_\lambda$. Lemma~\ref{lem:bounded} states that $\frac{\partial}{\partial \lambda} f_\lambda'$ is bounded for $\lambda\in[\frac{i}{N},\frac{i+1}{N}]$ for all $i=0,\ldots N-1$. Then Proposition~\ref{prop:continuation1} states that $H_*(f_{i/N},U)\cong H_*(f_{\frac{i+1}{N}},U)$, and hence $H_*(f_0,U)\cong H_*(f_1,U)$.
\end{proof}

\subsection{Local Morse cohomology and $E$-cohomological Conley index}

We can also consider Morse cohomology, which is defined through the usual dualization process at the chain level. As the coefficients are in the field $\mathbb{Z}_2$, the homology and cohomology groups are isomorphic by the universal coefficients theorem.  We denote local Morse cohomology by $H^*(f,U)$. 

\begin{theorem}
\label{theorem:isomorphism}
Let $U$ be an isolating neighborhood of the gradient flow $\phi$ of $f\in C^2(\mH,\mR)$. Assume that $f$  satisfies (B1) and (B2). Then,
$$H^*(f,U) \cong \mathrm{ch}_E^*(U,\phi).$$
\end{theorem}
\begin{proof}
By Proposition~\ref{prop:dense1} we can find a function $f'=\frac{1}{2}\langle Lx,x\rangle+b'(x)$ with $b'$ finite-dimensional support satisfying (B1) and (B2) such that the gradient flows $\phi$ and $\phi'$ are related by continuation through gradient flows of functions satisfying (B1). By Corollary~\ref{cor:morse_invariance} we have that $H^*(f,U)\cong H^*(f',U)$. Note that the critical points and connecting orbits of $f'$ lie all in $E_n$. Hence we can restrict $f'$ to $E_n$ and compute the finite-dimensional Morse cohomology $f'\bigr|_{E_n}$. We get
$$
H^*(f',U)\cong H^{*-\dim(E_n\cap E^-)}(f'\bigr|_{E_n},U\cap E_n),
$$
where the right hand side denotes the ordinary finite-dimensional Morse homology. This latter group is isomorphic to the finite-dimensional Conley index $\mathrm{ch}^{*-\dim (E_n\cap E^-)}(U\cap E_n,\phi'\bigr|_{E_n})$. We refer to~\cite{rotvandervorst} for a discussion of this isomorphism in this setting. The set $\tilde U=(U\cap E_n)\times B_1(E_n^\perp)$ is also an isolating neighborhood of $\phi'$ and $\mathrm{ch}^{*-\dim(E_n\cap E^-)}(U\cap E_n,\phi'\bigr|_{E_n})\cong \ch^*(\tilde U,\phi')$. By Theorem~\ref{thm:continuation} we have $\ch^*(U,\phi)\cong \ch^*(U,\phi')$ and $\ch^*(U,\phi')\cong \ch^*(\tilde U,\phi')$. By composing the isomorphisms we obtain the required result. 
\end{proof}

\begin{remark}
\label{rem:morsenonmorse}
In the next section we will need the local Morse homology for functions that do not necessarily satisfy the   Condition (B2). Let $U$ be an isolating neighborhood of the gradient flow of $f$ satisfying Condition (B1) on $U$. By the density of functions satisfying (B2), cf.~Proposition \ref{prop:dense1}, and the openness of isolation, cf.~Proposition~\ref{prop:close_fields} we can always continue the gradient flow of $f$ to the gradient flow of a function $f^\alpha$ which satisfies both (B1) and (B2). We define
$$
H_*(f,U):=\varprojlim \{H_*(f^\alpha,U);\Phi_*^{\beta\alpha}\},
$$
where the limit\footnote{\label{foot:iso}Every homology group $H_*(f^\alpha,U)$ is isomorphic to each other. We use the inverse limit to keep track of the isomorphisms which is important when one considers functoriality in local Morse homology, cf.~\cite{Rot:2014ku}.} runs over all $f^\alpha$ satisfying (B1)-(B2) whose gradient flow is related by continuation to the gradient flow of $f$. The maps $\Phi_*^{\beta\alpha}$ used to define the limit are the maps induced by continuation between $f^\alpha$ and $f^\beta$ which are defined in~Corollary \ref{cor:morse_invariance}. Thus we have defined the local Morse homology of a function that is not necessarily Morse.
\end{remark}

\section{Lyapunov functions and Morse-Conley-Floer homology}
\label{sec:mcf}

In~\cite{rotvandervorst} Lyapunov functions were used in a finite-dimensional context to define an intrinsic homology theory for (arbitrary) flows with Morse homological methods. Here we discuss an infinite-dimensional analogue in the context of $\ls$-flows. 
\begin{definition}
Let $S$ be an isolated invariant set of an $\ls$-flow $\phi$. A smooth function $f:\mH\rightarrow \mR$ is a smooth Lypanov function for $(\phi,S)$ if there exists an isolating neighborhood $U$ of $S$ such that
\begin{enumerate}
\item[(i)] $f\bigr|_S=\text{constant}$
\item[(ii)] $\frac{d}{dt}\bigr|_{t=0}f(\phi(t,x))<0$ for all $x\in U\setminus S$.
\end{enumerate}
Such a function $f$ is also referred to as a Lyapunov function for $\phi$ on $U$. The set of all Lyapunov functions for $\phi$ on $U$ is denoted $\lyap(\phi,U)$. If $f$ satisfies Condition (B1) in Section \ref{localMH} then $f$ is said to be an $\ls$-Lyapunov function. The set of $\ls$-Lyapunov functions for $\phi$ on $U$ is denoted by $\lyap_\ls(\phi,U)$. 
\end{definition}

The set of Lyapunov functions for $\ls$-flows is non-empty, but this does not directly imply that the set of all $\ls$-Lyapunov functions is non-empty. We conjecture it to be non-empty, but we will only need the following weaker statement is which we know is true by Reineck's Theorem~\ref{theorem:LSReineck}. In the continuation class of every $\ls$-flow there exist an $\ls$-flow which admits an $\ls$-Lyapunov function. 

Observe, by essentially the same arguments of Sections 2 and 3 of~\cite{rotvandervorst}, that if $U$ is an isolating neighborhood of $\phi$, then $U$ is an isolating neighborhood of the gradient flow of every $f\in \lyap_\ls(\phi,U)$. The set $\lyap_\ls(\phi,U)$ is convex which implies that the $\ls$-gradient flows of $f\in \lyap_\ls(\phi,U)$ are all related by continuation in $U$. 

We define Morse-Conley-Floer homology of an arbitrary isolated invariant set of an $\ls$-flow using Lyapunov functions as follows. 

\begin{definition}
Let $U$ be an isolating neighborhood of an $\ls$-flow $\phi$. Define the Morse-Conley-Floer homology to be
$$
HI_*(U,\phi):=\varprojlim\{H_*(f^\alpha,U);\Phi_*^{\beta\alpha}\}.
$$
The inverse limit runs over the local Morse homology of pairs $(f^\alpha,U)$, where $f^\alpha$ is an $\ls$-Lyapunov function of an $\ls$-flow $\psi^\alpha$ that is related by continuation to $\phi$ in $U$.
\end{definition}
The Morse-Conley-Floer homology is well-defined: Lyapunov functions are generally not Morse, but the local Morse homology of functions satisfying only (B1) is defined in Remark~\ref{rem:morsenonmorse}. Theorem~\ref{theorem:LSReineck} states that every $\ls$-flow is related by continuation to an $\ls$-gradient flow. An $\ls$-gradient flow clearly admits an $\ls$-Lyapunov function and we see that the set over which the inverse limit runs is non-empty. Any two $\ls$-gradient flows are related by continuation in $U$ through $\ls$-gradient flows and we obtain functorial continuation maps $\Phi_*^{\beta\alpha}$ between different choices $f^\beta$ and $f^\alpha$.

By the usual dualization process we obtain the Morse-Conley-Floer cohomology $HI^*$. The Morse-Conley-Floer cohomology is isomorphic to the $E$-cohomological Conley index. 

\begin{theorem}
Let $U$ be an isolating neighborhood of an $\ls$-flow $\phi$. Then 
$$
HI^*(U,\phi)\cong \ch^*(U,\phi).
$$
\end{theorem}
\begin{proof}
Let $\psi^\alpha$ be an $\ls$-flow, which admits an $\ls$-Lyapunov function in $U$ and which is related by continuation to $\phi$ in $U$. All continuation maps occuring in the inverse limit in the definition of the Morse-Conley-Floer cohomology are isomorphisms, cf. Footnote~\ref{foot:iso}, hence we see that $HI^*(U,\phi)$ is isomorphic to the local Morse cohomology $H^*(f^\alpha,U)$. Theorem~\ref{theorem:isomorphism} states that $H^*(f^\alpha,U)$ is isomorphic to the $E$-cohomological Conley index $\mathrm{ch}_E^*(U,\psi^\alpha)$. The flow $\psi^\alpha$ is related by continuation to $\phi$ in $U$ thus $\mathrm{ch}_E^*(U,\psi^\alpha)$ is isomorphic to $\mathrm{ch}_E^*(U,\phi)$ by Theorem~\ref{thm:continuation}. Therefore $HI^*(U,\phi)\cong \mathrm{ch}_E^*(U,\phi)$.
\end{proof}

\section*{Acknowledgements}
M. Izydorek, T.O. Rot and M. Starostka were partially supported by DAAD and MNISW PPP-PL Grant no. 57217076. M. Starostka was also supported by National Science Centre grant UMO-2015/17/N/ST1/02527.

\appendix
\section{E-cohomology}
\label{app:ecoh}
In this appendix we recall the definition and some properties of $E$-cohomology. For the details we refer the reader to~\cite{abbondandolo1997new}.

Let $\mH$ be a separable real Hilbert space with a splitting $\mH = E^+ \oplus E^-$ into two closed subspaces. Let $X$ be closed and bounded subset of $\mH$. For a finite-dimensional subspace $V$ of $E^-$ put 
$$ X_V := X \cap( E^+\oplus V).$$
Suppose we have another subspace $W$ such that $W = V \oplus U$ where $\dim U = 1$. We orient $U$ by picking a vector $u_0$ and define the positive and negative parts of $X_W$ by
$$X_W^+ = \{w \in X_W\,|\, \ip{w}{u_0} \geq 0\}, \quad X_W^- = \{w \in X_W\,|\, \ip{w}{u_0} \leq 0\}.$$

\begin{figure}[b]
\begin{overpic}[width=0.5\textwidth]{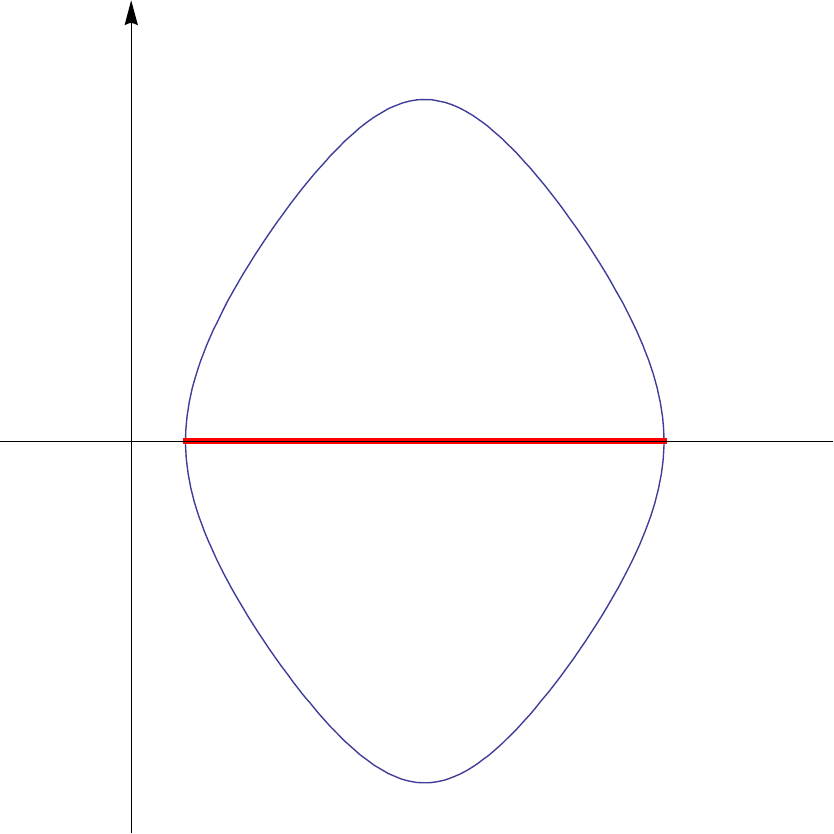}
 \put (9,94) {$U$}
 \put (95,40) {$V$}
 \put (50,50) {\textcolor{red}{$X_V$}}
 \put (57,70) {\textcolor{blue}{$X_W^+$}}
 \put (57,20) {\textcolor{blue}{$X_W^-$}}
 \put (70,85) {\textcolor{blue}{$X_W$}}
\end{overpic}
\caption{To compute the $E$-cohomology of a set $X$ we need to use the Mayer-Vietoris sequence of the triad $(X_W,X_W^+,X_W^-)$.}
\end{figure}
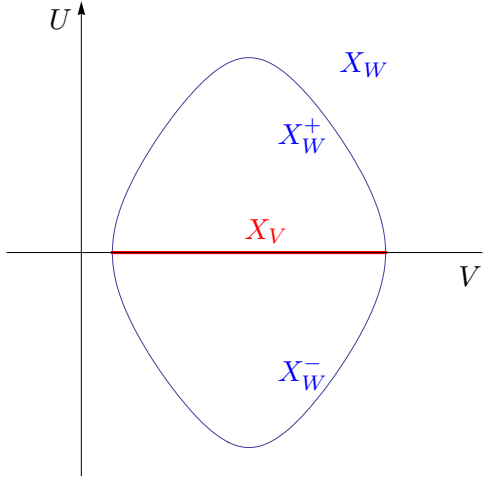
Note that $X_W^+ \cap X_W^- = X_V$ and therefore the Mayer-Vietoris sequence for a triad $(X_W,X_W^+,X_W^-)$ reads 
$$\ldots \rightarrow H^k(X_W^+) \oplus H^k(X_W^-) \to H^k(X_V) \xrightarrow{\Delta^k_{V,W}(X)} H^{k+1}(X_W) \to \ldots $$
Following Abbondandolo we use Alexander-Spanier cohomology above, but any cohomology theory satisfying the strong excision axiom would work.
\begin{definition}\label{def:HqE}
The $q$-th $E$-cohomology group of $X$ is defined as the direct limit
$$H^q_E(X) = \varinjlim_{V \subset E^-, \dim V < \infty} \{H^{q+\dim V}(X_V); \Delta_{V,W}(X)\}.$$
Analogously, we define $E$-cohomology groups $H^q_E(X,A)$ of closed and bounded pairs $(X,A)$.
\end{definition}
The $E$-cohomology groups are \textit{middle-dimensional}: Let $V$ be a subspace of $\mH$ such that the $E$-dimension of $V$
$$p = \dime V := \dim (V \cap E^+) - \codim (V + E^+)$$
is finite and let $X = S(V)$ be a unit sphere in $V$. Direct computations lead to
$$H^{q}_E(X) = \left\{
  \begin{array}{ll}
    \coef & \hbox{if $q = p - 1$;} \\
    0 & \hbox{otherwise.}
  \end{array}
\right.$$

Having defined $E$-cohomology groups, let us recall the definitions of the $E$-morphisms.
\begin{definition}
Let $(X,A)$ and $(Y,B)$ be closed and bounded pairs.
A continuous map $\Psi:(X,A) \to (Y,B)$ is an $E$-morphism  if it has the form 
$$\Psi(x) = Mx + U(x)$$
where $U:X \to \mH$ maps bounded sets into precompact sets and $M$ is a linear automorphism of $\mH$ such that $ME^- = E^-$.
\end{definition}

\begin{remark}\label{remark:extend}
In fact, the above class of morphisms can be extended to the class of continuous maps $\Psi\:(X,A)\to (Y,B)$ of the form $\Psi(x)=M(x) x+U(x)$, where $U$ is as above, $M:X\rightarrow GL(\mH)$ has compact image and $M(x)E^-=E^-$.

\end{remark}

As usual, an $E$-homotopy is a homotopy in the $E$-category. The $E$-cohomology satisfies generalized Eilenberg-Steenrod axioms. To be more precise, it satisfies the axioms with the set of morphisms restricted to $E$-morphisms and with the dimension axiom stating that the unit sphere in $E^-$ has non-trivial cohomology exactly in dimension $q = -1$. Directly by the homotopy invariance of the $E$-cohomology we have the following two lemmas. The maps are sometimes referred to as \textit{deformation retracts in the weak sense}.

\begin{lemma} \label{lemma:deformation_retract_2}
  Let $A\subset Y\subset X$ and let $\Phi\:[0,1]\times(X,A)\to(X,A)$ be an $E$-homotopy, such that $\Phi(0,\,\cdot\,)=\mathrm{id}_{(X,A)}$, $\Phi(1,X)\subset Y$ and $\Phi(\lambda,Y)\subset Y$ for all $\lambda\in[0,1]$ . Then the inclusion $(Y,A)\hookrightarrow(X,A)$ induces an isomorphism $\HE{*}{X}{A}\xrightarrow{\cong}\HE{*}{Y}{A}$.
\end{lemma}

\begin{lemma} \label{lemma:deformation_retract}
  Let $A\subset Y\subset X$ and $\Phi\:[0,1]\times(X,Y)\to(X,Y)$ be an $E$-homotopy, such that $\Phi(0,\,\cdot\,)=\mathrm{id}_{(X,Y)}$, $\Phi(1,Y)\subset A$ and  $\Phi(\lambda,A)\subset A$ for all $\lambda\in[0,1]$. Then the inclusion $(X,A)\hookrightarrow(X,Y)$ induces an isomorphism $\HE{*}{X}{Y}\xrightarrow{\cong}\HE{*}{X}{A}$.
\end{lemma}

The $E$-cohomology satisfies also the following continuity property.

\begin{lemma}[Continuity property]\label{lemma:continuity}
Let $(X,A)$ be closed and bounded and let $\{U^m,V^m\}$ be a sequence of closed and bounded pairs such that
\begin{enumerate}
\item $U^m \subset U^n$ and $V^m \subset V^n$ if $n \leq m$;
\item $\bigcap_{m \in \mN}U^m = X$ and $\bigcap_{m \in \mN}V^m = A$.
\end{enumerate}
Then the direct limit 
$$\varinjlim_{m \in \mN} H^*_E(i^m): \varinjlim_{m \in \mN} \{H^*_E(U^m,V^m);H^*_E(j^{m,n})\} \to H^*_E(X,A)$$
is an isomorphism.
\end{lemma}

\bibliographystyle{alpha}
\bibliography{ref}

\end{sloppypar}
\end{document}